\documentclass {article}
\usepackage{amsmath, amssymb, amsthm, amsfonts, amsxtra, latexsym,
amscd,  pb-diagram,  graphics,rotating,setspace}
\usepackage [all]{xy}

\usepackage{enumerate}
\theoremstyle{plain}

\newcommand{\A}{\mathcal{A}}
\newcommand{\tx}{\otimes }
\newcommand{\ts}{\oplus}
\newcommand{\ri}{\rightarrow }

\newcommand{\Fa}{\Breve{F}}

\newcommand{\Lo}{\widehat{L}}

\def\tx{\otimes}
\def\ts{\oplus}
\def\Fn{\widetilde F}

\newcommand{\Lh}{\frak{L}}
\newcommand{\Rh}{\frak{R}}

\newtheorem{thm}{\bf Theorem}
\newtheorem{lem}[thm]{\bf  Lemma} % Lemma
\newtheorem{pro}[thm]{\bf  Proposition} % Proposition
\newtheorem{hq}[thm]{\bf Corollary} % Corollary

\theoremstyle{definition}

 \DeclareMathOperator{\Coker}{Coker}
\DeclareMathOperator{\Aut}{{Aut}} \DeclareMathOperator{\Ob}{{
Ob}}\DeclareMathOperator{\End}{{End}}

\begin{document}

\centerline{\bf \large Crossed Bimodules over Rings and Shukla
Cohomology} \vspace{10pt}

\centerline{\bf Nguyen Tien Quang}

\centerline{ Department of Mathematics, Hanoi National University of
Education,  Hanoi, Vietnam} \centerline{cn.nguyenquang@gmail.com}
\vspace{5pt}
 \centerline{\bf Pham Thi Cuc}
 \centerline{ Natural Science Department,
Hongduc University,  Thanhhoa, Vietnam }

\centerline{ cucphamhd@gmail.com}

\begin{abstract}
In this paper we present some applications of Ann-category theory to
classification of crossed bimodules over rings, classification of
ring extensions  of the type of a crossed bimodule.
\end{abstract}
%%%%% end %%%%%%%%%%%

%%%% Start %%%%%%
\section{Introduction}
Crossed modules over groups were introduced by J. H. C. Whitehead
\cite{White49}. A crossed module over a group $G$ with kernel a
$G$-module $M$ represents an element in the cohomology
 $H^3(G,M)$ \cite{Mac-Wh}. The results  on group extensions of the type of a crossed module
were also represented by the cohomology of groups  \cite {Br94}.

Later,  H-J. Baues \cite{B2002} introduced crossed modules over $\bf
k$-algebras. Crossed modules over $\bf k$-algebras which are  $\bf
k$-split with the same kernel $M$ and cokernel $B$ were classified
by Hochschild cohomology $H^3_{Hoch}(B,M)$ \cite{BM2002}.

In \cite {P2004} the field $\bf k$ is replaced  by a commutative
ring $\mathbb K$, and crossed modules over $\mathbb K$-algebras were
called {\it crossed bimodules}. In particular, if $\mathbb K=\mathbb
Z$ one obtains crossed bimodules  over rings.

Crossed modules over groups can be defined over rings in a different
way under the name of  {\it E-systems}.  The notion of an E-system
is weaker than that of a crossed bimodule over rings.

Crossed modules over groups are often studied in the form of
$\mathcal G$-groupoids  \cite{Br76}, or  strict  2-groups
\cite{Baez}. From this point, we represent E-systems in the form of
strict Ann-categories (also called strict 2-rings). Hence, one can
use the results on Ann-category  theory to study crossed bimodules
over rings.

The plan of this paper is, briefly, as follows.  Section 2 is
dedicated to review definitions and some basic facts concerning
Ann-categories. In Section 3, we introduce the concept of an
E-system and prove that there is an isomorphism between  the
category of regular E-systems and that of crossed bimodules over
rings. The relation among these concepts and crossed $\bf C$-modules
in the sense of T. Porter\cite{TP} is also discussed.  The next
section is devoted to showing a categorical equivalence of the
category of E-systems and a subcategory of the category of strict
Ann-categories, which is an extending of the result of  R. Brown and
C. Spencer \cite{Br76}.

The group extensions of the type of a crossed module were dealt with
by R. Brown and O. Mucuk \cite{Br94}. The similar results for
$\partial$-extensions by an algebra $R$ were done by H-J. Baues and
T. Pirashvili \cite{P2004} in a particular case.
 In Section 5 we solve this problem for ring extensions of the type of an
 E-system by  Shukla cohomology groups. Our classification result contains
  the result in \cite{P2004} when $R$ is a ring.
\section{Ann-categories}
%{Preliminaries}
We state a minimum of necessary concepts and facts of Ann-categories
and Ann-functors (see \cite{Q1}).

A  {\it Gr-category} (or a {\it categorical group}) is a monoidal
category in which all objects are invertible and the background
category is a groupoid. A {\it Picard } category (or a {\it
symmetric} categorical group) is a Gr-category equipped with a
symmetry constraint which is compatible with associativity
constraint.

\noindent{\bf Definition 1.} An {\it Ann-category} consists of

\noindent $\mathrm{i)}$ a category  $\A$ together with two
bifunctors $\ts, \tx :\A \times \A \rightarrow \A$;

\noindent $\mathrm{ii)}$ a fixed object $0 \in \Ob(\A)$ together
with natural isomorphisms  ${\bf a_+, c,g,d}$ such that $(\A,
\ts,{\bf a_+,c},(0,{\bf g,d}))$ is a Picard category;

\noindent $\mathrm{iii)}$ a fixed object $1 \in \Ob(\A)$ together
with natural isomorphisms  ${\bf a,l,r}$ such that $(\A,\tx,{\bf
a},(1,{\bf l,r}))$ is a monoidal category;

\noindent $\mathrm{iv)}$ natural isomorphisms $\Lh,\Rh$ given by
\[\begin{array}{ccccc}
\Lh_{A,X,Y}: & A\tx (X\ts Y) &\longrightarrow & (A\tx X)\ts (A\tx Y),\\
\Rh_{X,Y,A}: & (X\ts Y)\tx A &\longrightarrow & (X\tx A)\ts (Y\tx A)
\end{array}\]
such that the following conditions hold:

\noindent (Ann - 1) for  $A\in \Ob(\A)$, the pairs $(L^A, \Breve
L^A), (R^A, \Breve R^A)$ defined by
\[\begin{array}{lclclclcc}
L^A & = & A\tx - & \qquad\qquad & R^A & = & -\tx A \\
\breve L^A_{X,Y} & = & \Lh_{A,X,Y} & \qquad\qquad & \breve R^A_{X,Y}
& = & \Rh_{X,Y,A}
\end{array}\]
are $\ts$-functors which are compatible with ${\bf a_+}$ and ${\bf
c}$;

\noindent (Ann - 2) for all $A,B,X,Y \in \Ob(\A),$ the following
diagrams commute  {\scriptsize
\[\begin{diagram}
\node{(AB)(X\ts Y)}\arrow{s,l}{\Breve L^{AB}}\node{A(B(X\ts
Y))}\arrow{w,t} {{\bf a}_{A,B,X\ts Y}}\arrow{e,t}{id_A\tx \Breve
L^B}\node{A(BX \ts BY)}\arrow{s,r}
{\Breve L^A}\\
\node{(AB)X\ts (AB)Y}\node[2]{A(BX)\ts A(BY)}\arrow[2]{w,t}{{\bf
a}_{A,B,X}\ts {\bf a}_{A,B,Y}}
\end{diagram}\]
%%%%%%%%%%%%%%%%%%%%%%%%%%%%%%%%%%%%%%%%%%%%%%%%%%%%%%%%%%%%%%%%%%%%%%%%%%%%%%%%%%%%%%%%

\[\begin{diagram}
\node{(X\ts Y)(BA)}\arrow{s,l}{\Breve R^{BA}}\arrow{e,t}{{\bf
a}_{X\ts Y,B,A}} \node{((X\ts Y)B)A}\arrow{e,t}{\Breve R^B\tx id_A}
\node{(XB \ts YB)A}\arrow{s,r}{\Breve R^A}\\
\node{X(BA)\ts Y(BA)}\arrow[2]{e,t}{{\bf a}_{X,B,A}\ts {\bf
a}_{Y,B,A}}\node[2]{(XB)A \ts (YB)A}
\end{diagram}\]
%%%%%%%%%%%%%%%%%%%%%%%%%%%%%%%%%%%%%%%%%%%%%%%%%%%%%%%%%%%%%%%%%%%%%%%%%%%%%%%%%%%%%%%%
\[\begin{diagram}
\node{(A(X\ts Y)B}\arrow{s,l}{\Breve L^A \tx id_B} \node{A((X\ts
Y)B)}\arrow{w,t}{{\bf a}_{A,X\ts Y,B}}\arrow{e,t}{id_A\tx \breve
R^B}
\node{A(XB \ts YB)}\arrow{s,r}{\Breve L^A}\\
\node{(AX\ts AY)B}\arrow{e,t}{\breve R^B}\node{(AX)B \ts (AY)B}
\node{A(XB) \ts A(YB)}\arrow{w,t}{{\bf a}\ts {\bf a}}
\end{diagram}\]
%%%%%%%%%%%%%%%%%%%%%%%%%%%%%%%%%%%%%%%%%%%%%%%%%%%%%%%%%%%%%%%%%%%%%%%%%%%%%%%%%%%%%%%%
\[\begin{diagram}
\node{(A\ts B)X\ts (A\ts B)Y}\arrow{s,l}{\Breve R^X \ts \breve R^Y}
\node{(A\ts B)(X\ts Y)}\arrow{w,t}{\breve L^{A\ts
B}}\arrow{e,t}{\breve R^{X\ts Y}}
\node{A(X \ts Y)\ts B(X\ts Y)}\arrow{s,r}{\Breve L^A \ts\Breve L^B }\\
\node{(AX\ts BX)\ts (AY\ts BY)}\arrow[2]{e,t}{{\bf v}}\node[2]{(AX
\ts AY)\ts (BX \ts BY)}
\end{diagram}\]}
%%%%%%%%%%%%%%%%%%%%%%%%%%%%%%%%%%%%%%%%%%%%%%%%%%%%%%%%%%%%%%%%%%%%%%%%%%%%%%%%%%%%%%%%
where ${\bf v} = {\bf v}_{U,V,Z,T}:(U\ts V)\ts (Z\ts T)
\longrightarrow (U\ts Z)\ts (V\ts T)$ is a unique  morphism
constructed from $\oplus, {\bf a_+, c}, id$ of the symmetric
monoidal category $(\A,\ts)$;

\noindent (Ann - 3) for the unit $1 \in \Ob(\A)$ of the operation
$\tx$, the following diagrams commute {\scriptsize
\[\begin{diagram}
\node{1(X\ts Y)} \arrow[2]{e,t}{\breve L^{1}}\arrow{se,b}{{\bf
l}_{X\ts Y}} \node[2]{1 X\ts 1 Y}\arrow{sw,b}{{\bf l}_X\ts {\bf
l}_Y} \node{(X\ts Y)1}\arrow[2]{e,t}{\breve R^{1}}\arrow{se,b}{{\bf
r}_{X\ts Y}}
\node[2]{X1 \ts Y1}\arrow{sw,b}{{\bf r}_X\ts {\bf r}_Y}\\
\node[2]{X\ts Y} \node[3]{X\ts Y.}
\end{diagram}\]}
%%%%%%%%%%%%%%%%%%%%%%%%%%%%%%%%%%%%%%%%%%%%

%%%%%%%%%%%%%%%%%%%%%%%%%%%%%%%%%%%%%%%%%%%
%
An Ann-category $\mathcal A$ is {\it regular} if its symmetry
constraint satisfies the condition $\mathbf{c}_{X,X}=id$, and  {\it
strict} if all of its constraints are identities.

\noindent{\bf Example 1.} Let $\A=(\A,\oplus)$ be a Picard category
whose unity and associativity constraints are identities.  Denote by
  End$(\A)$ a category whose objects are symmetric monoidal functors
  from
$\A$ to $\A$ and whose morphisms are $\oplus$-morphisms. Then,
End$(\A)$ is a Picard category together with the operation $\oplus$
on monoidal functors and on morphisms. In this $\oplus$-category,
the  unity and associativity constraints are identities, the
commutativity constraint is given by
$$(\mathbf{c}^\ast_{F,G})_X=\mathbf{c}_{FX,GX},\;\;X\in\Ob(\A),\;F,G\in\End(\A).$$
The operation $\otimes$ on End$(\A)$ is naturally defined being the
composition of functors. Then, End$(\A)$ together with two
operations $\oplus,\otimes$ is an Ann-category in which the left
distributivity constraint is given by
$$(\mathfrak{L
}^\ast_{F,G,H})_X=\breve{F}_{GX,HX},\;\;X\in\Ob(\A),$$ and other
constraints are identities (for details, see \cite{Q07}).

\noindent{\bf Example 2.} Let $R$ be a ring with an unit and  $M$ be
an $R$-bimodule. The pair $\mathcal{I}=(R,M)$ is a category whose
objects are elements of $R$ and whose morphisms are automorphisms
$(r,a):r\ri r$, $r\in R, a\in M$. The composition of morphisms is
given by the addition in $M$. Two operations $\oplus$ and $\otimes$
on $\mathcal I$ is defined by
$$x\oplus y=x+y,\;(x,a)+(y,b)=(x+y,a+b),$$
$$x\otimes y=xy,\;(x,a)\otimes(y,b)=(xy, xb+ay).$$
The constraints of $\mathcal I$ are identities, except for left
distributivity and commutativity constraints which are given by
$$\mathcal L_{x,y,z}=(\bullet,\lambda(x,y,z)):x(y+z)\ri xy+xz,$$
$$\mathbf{c}^+_{x,y}=(\bullet,\eta(x,y)):x+y\ri y+x,$$
where $\lambda:R^3\ri M,\eta:R^2\ri M$ are functions satisfying the
appropriate coherence conditions.

 Here are standard
consequences of the axioms of an Ann-category.

\begin{lem}
For every Ann-category $\A$ there exist uniquely isomorphisms
$$
\hat L^A: A\tx 0 \rightarrow 0, \qquad \hat R^A:  0\tx A \rightarrow
0,
$$
where $A\in\Ob(\A)$, such that $\ts$-functors $(L^A,\breve L, \hat
L^A)$ and $(R^A,\breve R, \hat R^A)$ are compatible with unit
constraints $(0,{\bf g,d}).$
\end{lem}
%%%%%%%%%%%%%%%%%%%%%%%%%%%%%%%%%
It is easy to see that if $(\A,\oplus)$ and $(\A',\oplus)$ are
Gr-categories, then every $\oplus$-functor $(F,\breve
F):\A\rightarrow \A'$, which is compatible with associativity
constraints, is a monoidal functor. Thus, we state the following
definition.

\noindent{\bf Definition 2.} Let $\A$ and $\A'$ be Ann-categories.
An {\it Ann-functor} $(F,\breve F, \Fn, F_\ast):\A\ri \A'$ consists
of a functor $F:\A\rightarrow \A'$, natural isomorphisms
$$\breve F_{X,Y}:F(X\oplus Y)\rightarrow F(X)\oplus F(Y),\ \widetilde{F} _{X,Y}:F(X\otimes
 Y)\rightarrow F(X)\otimes F(Y),$$
 and an isomorphism $F_\ast:F(1)\rightarrow 1'$ such that $(F,\breve F)$ is a symmetric monoidal functor
for the operation  $\ts$, $(F,\Fn, F_\ast)$ is a monoidal functor
for the operation $\tx$, and the following  diagrams commute
{\scriptsize
\[\begin{diagram}
\node{F(X(Y\ts Z))}\arrow{s,l}{F(\Lh)}
\arrow{e,t}{\Fn}\node{FX.F(Y\ts Z)}\arrow{e,t}{id
\tx\Fn}\node{FX(FY\ts FZ)}
\arrow{s,r}{\Lh'}\\
\node{F(XY\ts XZ)}\arrow{e,t}{\breve F}\node{F(XY)\ts F(XZ)}\arrow{e,t}{\Fn\ts\Fn}\node{FX.FY\ts FX.FZ}%\tag{2.1a}
\end{diagram}\]
\[\begin{diagram}
\node{F((X\ts Y)Z)}\arrow{s,l}{F(\Rh)}\arrow{e,t}{\Fn}\node{F(X\ts Y).FZ}\arrow{e,t}{\breve F\tx id}\node{(FX\ts FY)FZ}\arrow{s,r}{\Rh'}\\
\node{F(XZ\ts YZ)}\arrow{e,t}{\breve F}\node{F(XZ)\ts F(YZ)} \arrow{e,t}{\Fn\ts\Fn}\node{FX.FZ\ts FY.FZ.}%\tag{2.1b}
\end{diagram}\]}
%%%%%%%%%%%%%%%%%%%%5

These diagrams are  called the {\it compatibility} of the functor
$F$ with the distributivity constraints.

\indent An {\it Ann-morphism} (or a {\it homotopy})
$$\theta : (F, \Fa, \widetilde{F}, F_{\ast})\rightarrow (F', \breve{F'}, \widetilde{F}', F'_{\ast})$$
between Ann-functors is an  $\oplus$-morphism, as well as an
$\otimes$-morphism.

 If there exist an Ann-functor $(F',
\breve{F'}, \widetilde{F}', F'_{\ast}):\A'\rightarrow \A$ and
Ann-morphisms $F'F\stackrel{\sim}{\rightarrow} id_{\A}, \
FF'\stackrel{\sim}{\rightarrow} id_{\A'}$, we say that
$(F,\Fa,\widetilde{F}, F_{\ast})$ is an {\it Ann-equivalence}, and
$\A$, $\A'$ are {\it Ann-equivalent}.

%%%%%%%%%%%%%%%%%%%%%%%%%%%%%%%%%%%%%%%%%%%%%
For an Ann-category $\A$, the set $R=\pi_0\A$ of isomorphism classes
of the objects  in $\A$ is a ring with two operations $+,\times$
induced by the functors $\ts,\tx$ on $\A$, and the set $M=\pi_1\A =
\Aut(0)$ is a group with the composition denoted by $+$. Moreover,
$M$ is a $R$-bimodule with the actions
  \begin{align}\label{ct0a}
sa=\lambda_X(a), \quad as=\rho_X(a),\nonumber
\end{align}
for $X\in s, s\in\pi_0\A, a\in\pi_1\A$, and $\lambda_X,\rho_X$
satisfy
$$\lambda_X(a)\circ \hat L^X=\hat L^X\circ(id\tx {\bf a}) :X.0\ri 0,$$
$$\rho_X(a)\circ \hat R^X=\hat R^X\circ({\bf a}\tx id) :0.X\ri 0.$$

We recall briefly some main facts of the construction of the reduced
Ann-category $S_{\A}$
 of   $\A$ via the structure transport  (for details, see \cite{Q1}).
 The objects of
$S_{\A}$ are the elements of the ring $\pi_0\A$. A morphism is an
 automorphism $(s,a): s\rightarrow s,\ s\in \pi_0\A, a\in
\pi_1\A$. The composition of morphisms is given by
\[(s,a)\circ (s,b)=(s,a+b).\]

For each  $ s\in \pi_0\A$, choose an object $X_s\in\Ob(\mathcal A)$
such that $X_0=0, X_1=1$,
  and choose an isomorphism $i_X: X\rightarrow X_s$ such that $i_{X_s}=id_{X_s}$.
We obtain two functors
\begin{align*}  \begin{cases}
G:\mathcal A\ri S_{\mathcal A}\\
G(X)=[ X]=s\\
G(X \stackrel {f}{\ri}Y)=(s,\gamma_{X_s}^{-1}(i_Yfi_X^{-1})),
\end{cases}\qquad\qquad
\begin{cases}
H: S_{\mathcal A}\ri  \mathcal A\\
H(s)=X_s\\
H(s,u)=\gamma_{X_s}(u), %\label{eq1}
\end{cases}\end{align*}
\noindent  for $X,Y\in s$, $f: X\rightarrow Y$, and $\gamma_X$
\begin{equation}\label{ld}\gamma_X({\bf a})={\bf g}_X\circ({\bf
a}\ts id)\circ{\bf g}^{-1}_X. \end{equation}

\noindent Two operations on $S_{\A}$ are given by
\begin{eqnarray*}
s\ts t&=&G(H(s)\ts H(t))=s+t,\\
s\tx t&=& G(H(s)\tx H(t))=st,\\
(s,a)\ts (t,b)&=&G(H(s,a)\ts H(t,b))=(s+t,a+b),\\
(s,a)\tx (t,b)&=&G(H(s,a)\tx H(t,b))=(st, sb+at),
\end{eqnarray*}
for $s,t\in \pi_0\A$, $a,b\in \pi_1\A$. Clearly, they do not depend
on the choice of the representative $(X_s, i_X).$

The constraints in $S_{\A}$ are defined by  sticks. A {\it stick} of
$\A$ is a representative $(X_s, i_X)$ such that
\begin{eqnarray*}
i_{0\oplus X_t}={\bf g}_{X_t},\qquad  i_{X_s\oplus 0}={\bf d}_{X_s},&\\
i_{1\otimes X_t}={\bf l}_{X_t}, \qquad i_{X_s\otimes 1}={\bf
r}_{X_s},& i_{0\otimes X_t}=\widehat{R}^{X_t},\quad i_{X_s\otimes
0}=\Lo^{X_s}.
\end{eqnarray*}

 The unit constraints in $S_{\A}$ are $(0, id,id)$ and $(1,id, id)$.
The family of the rest ones, $h=(\xi, \eta, \alpha, \lambda, \rho),$
is defined by the compatibility of the constraints ${\bf a_+, c, a},
\frak{L}, \frak{R}$ of $\A$  with the functor $H$ and isomorphisms
\begin{align}
\breve{H}=i^{-1}_{X_s\oplus X_t}, \widetilde{H}=i^{-1}_{X_s\otimes
X_t}.\label{eq2}
\end{align}
Then $(H, \breve{H}, \widetilde{H}):S_{\A}\ri \A$ is an
Ann-equivalence. Besides, the functor $G: \A\rightarrow S_{\A}$
together with isomorphisms
\begin{align}
\breve{G}_{X,Y}=G(i_X\ts i_Y),\quad \widetilde{G}_{X,Y}=G(i_X\tx
i_Y) \label{eq4}
\end{align}
is also an Ann-equivalence.
 We refer to $S_{\A}$ as an Ann-category of {\it type} $(R,M),$
 and $(H, \breve H,\widetilde{H})$, $(G, \breve G,\widetilde{G})$ are {\it canonical
 } Ann-equivalences.
The family of constraints $h=(\xi, \eta, \alpha, \lambda, \rho)$ of
$S_{\A}$ is called a {\it structure} of the Ann-category of type
$(R, M)$.

Mac Lane \cite{Mac2} and Shukla \cite{Sh} cohomomology groups at low
dimensions are used to classify Ann-categories and regular
Ann-categories, respectively. A structure $h$ of the Ann-category
$S_\mathcal A$ is an element in the group of Mac Lane 3-cocycles
$Z^3_{MacL}(R,M)$. In the case when $\mathcal A$ is regular, $h\in
Z^3_{Shu}(R,M)$.
\begin{pro}[Proposition 11 \cite{Q1}\label{md22}]
Let  $\A$ and $\A'$ be Ann-categories.

\noindent $\mathrm{i)}$ Every Ann-functor
$(F,\Fa,\widetilde{F}):\A\rightarrow \A'$ induces an Ann-functor
$S_F:S_{\A}\rightarrow S_{\A'}$ of type $(p, q),$ where
$$p=F_0:\pi_0\A\ri\pi_0\A',\ [X] \mapsto [FX],$$
$$q=F_1:\pi_1\A\ri\pi_1\A',\ u \mapsto \gamma_{F0}^{-1}(Fu),$$
for $\gamma$ is a map given by the relation $\mathrm{(\ref{ld})}$.

\noindent $\mathrm{ii)}$ $F$ is an equivalence if and only if $F_0,
F_1$ are isomorphisms.

\noindent $\mathrm{iii)}$ The Ann-functor  $S_F$ satisfies

$$S_F=G'\circ F\circ H,$$
\noindent where $H, G'$ are canonical Ann-equivalences.

\end{pro}

Let $\mathcal S=(R,M,h), \mathcal S'=(R',M',h')$ be Ann-categories.
Since $\Fa_{x,y}=(\bullet, \tau(x,y)),$ $
\widetilde{F}_{x,y}=(\bullet, \nu(x,y)),$
 then $g_F=(\tau, \nu)$ is a pair of maps {\it associated} with $(\Fa,
\widetilde{F}),$ we thus can regard an Ann-functor $F:\mathcal
S\rightarrow \mathcal S'$ as a triple $(p, q, g_F).$ It follows from
the compatibility of  $F$ with the constraints that
\begin{align*}
q_{\ast}h-p^{\ast}h'=\partial (g_F), %\label{7}
\end{align*}
where $q_{\ast}, p^{\ast}$ are canonical homomorphisms,
$$Z^{3}_{MacL}(R, M)\stackrel{q_{\ast}}{\longrightarrow}Z^{3}_{MacL}(R, M')\stackrel{p^{\ast}}{\longleftarrow}
 Z^{3}_{MacL}(R', M').$$
  Further, two Ann-functors  $(F,g_F), (F',g_{F'})$  are homotopic if and only if $F=F'$,
 that is, they are the same type of $(p,q)$, and there exists a function  $t:R\rightarrow M'$
  such that
  $g_{F'}=g_F+\partial t.$

We denote by
 $$\mathrm{Hom}_{(p, q)}^{Ann}[\mathcal S, \mathcal S']$$
  the set of homotopy classes of Ann-functors of type $(p,q)$
 from $\mathcal S$ to $\mathcal S'.$

Let $F:\mathcal S\rightarrow \mathcal S'$ be an Ann-functor of type
$(p,q)$, then the function
\begin{align}
k=q_{\ast}h-p^{\ast}h'\in Z^3_{MacL}(R,M') \label{eq5}
\end{align}
 is called an {\it  obstruction} of $F$.

\begin{thm}[Theorem 4.4, 4.5 \cite{Q10}]\label{s63}
A functor  $F:\mathcal S\rightarrow \mathcal S' $ of type $(p, q)$
is an $Ann$-functor  if and only if its obstruction $\overline{k}$
vanishes in $H_{MacL}^3(R, M')$. Then, there exists a bijection
\begin{align*}  \mathrm{Hom}_{(p, q)}^{Ann}[\mathcal S, \mathcal S']\leftrightarrow H^2_{MacL}(R, M')(=H^2_{Shu}(R,M')). %\label{eq6}
\end{align*}

\end{thm}
%%%%%%%%%%%%%%%%%%%%%%
%\newpage
\section{ Crossed bimodules over rings and regular E-systems}
The results on crossed bimodules can be found in
\cite{B2002,BM2002,P2004,P2008}. We shall show a characteristic of
crossed bimodules when the base ring $\mathbb K$ is the ring of
integers $\mathbb Z$. Based on this characteristic, we can establish
the relation between crossed bimodules over rings and Ann-category
theory in the next section.

 \noindent{\bf Definition}
\cite{P2008}{\bf .} A \emph{crossed bimodule} is a triple $(B,D,d)$,
where $D$ is an associative $\mathbb K$-algebra, $B$ is a
$D$-bimodule and $d: B\rightarrow D$ is a homomorphism of
$D$-bimodules such that
\begin{equation} d(b)b'=bd(b'), \ b,b'\in B.\label{eq14}\end{equation}

A {\it morphism} $(k_1,k_0):(B,D,d)\ri (B',D',d')$ of crossed
bimodules  is a pair  $k_1:B\ri B'$, $k_0:D\ri D'$, where $k_1$ is a
group homomorphism, $k_0$ is a $\mathbb K$-algebra homomorphism such
that
\begin{equation}\label{cb1}
k_0d=d'k_1
\end{equation}
and
\begin{equation}\label{cb2}
k_1(xb)={k_0(x)}k_1(b),\; k_1(bx)= k_1(b){k_0(x)},
\end{equation}
 for all $x\in D,b\in B$.

 The condition \eqref{cb2} shows that $k_1$ is a homomorphism of $D$-bimodules, where $B'$
  is a $D$-bimodule with the action $xb'=k_0(x)b', \ b'x=b'k_0(x)$.

%%%%%%%%%%%%%%%%%
Below, the base ring  $\mathbb K$ is the ring of integers $\mathbb
Z$, and a crossed bimodule $(B,D,d)$ is caled  a {\it crossed
bimodule over rings}. Thus, $D$ is a ring with unit.

In order to introduce the concept of an E-system, we now recall some
 terminologies  due to  Mac Lane \cite{Mac2}. The set of all
bimultiplications of a ring $A$ is a ring denoted by $M_A.$ For each
 element $c$ of $A,$ a bimultiplication $\mu_c$ is defined by
\begin{align*}  \mu_ca=ca,a\mu_c=ac,a\in A \label{eqtt}  \end{align*}
we call $\mu_c$ an {\it inner bimultiplication.} Then $C_A=\{c\in
A|\mu_c=0\}$ is called the {\it bicenter} of $A.$

The bimultiplications $\sigma$ and $\tau$ are  {\it permutable} if
for every  $ a \in A $,
\begin{equation}\label{eq133}
\sigma(a\tau)=(\sigma a)\tau,\;\ \tau(a\sigma)=(\tau a)\sigma.
\end{equation}

%%%%%%%%%%%%%%%%%%%%%%%%%%%%%%%%%%%%%%%%%%%%%%%%%%%%%%

We now introduce  the main concept of the present paper which can be
seen as a version of the concept of a crossed module over rings.

\noindent {\bf Definition 3.} An $E$-{\it system} is a quadruple
$(B,D,d,\theta)$, where $d:B\ri D,\;\theta:D\ri M_B$ are the ring
homomorphisms such that the following diagram commutes
\begin{align}  \begin{diagram}\label{eq12}
\node{B}\arrow{se,b}{\mu}\arrow[2]{e,t}{d}\node[2]{D}\arrow{sw,b}{\theta}\\
\node[2]{M_B}
\end{diagram}\end{align}
and the following relations hold for all $x\in D, b\in B,$
 \begin{align}  d(\theta_xb)=x.d(b),\;\ d(b\theta_x)=d(b).x. \label{eq13}
\end{align}
\indent An E-system $(B,D,d,\theta )$ is  \emph{regular} if $\theta
$ is a 1-homomorphism (a homomorphism carries the identity to the
identity), and the elements of $\theta(D)$ are permutable.

 A {\it morphism} $(f_1,f_0):(B,D,d,\theta)\ri
(B',D',d',\theta')$ of E-systems consists of  ring homomorphisms
$f_1:B\ri B'$, $f_0:D\ri D'$ such that
\begin{equation}\label{gr1}
f_0d=d'f_1
\end{equation}
and  $f_1$ is an {\it operator homomorphism}, that is,
\begin{equation}\label{eq15}
f_1(\theta_xb)=\theta'_{f_0(x)}f_1(b),\; f_1(b\theta _x)=
f_1(b)\theta '_{f_0(x)}.
\end{equation}

In this paper, an E-system $(B,D,d,\theta)$ is sometimes denoted by
 $B\stackrel{d}{\rightarrow}D$, or $B\rightarrow D$.

\noindent {\bf Example 3.}
 If $B$ is a two-sided  ideal in $D$, then $(B,D,d,\theta)$ is a
regular E-system, where $d$ is an inclusion, $\theta:D\ri M_B$ is
given by the {\it bimultiplication type}, that is,
$$\theta_xb=xb,\;b\theta_x=bx,\;\;x\in D, b\in B.$$
\noindent {\bf Example 4.} Let $D$ be a ring, $B$ be a $D$-bimodule,
${\bf 0}:B\ri D$ is the zero homomorphism of $D$-bimodules. $B$  can
be considered as a ring with zero multiplication defined by
$b.b'={\bf 0}(b)b'=b{\bf 0}(b')=0$, for all $b,b'\in B$.
 Then, $(B,D,{\bf 0},\theta)$ is a regular E-system, where $\theta$ is
 given by the action of $D$-bimodules.

\noindent {\bf Example 5.} Let $B$ be a ring, $M_B$ be the ring of
bimultiplications of
 $B$, and
 $\mu:B\ri M_B$ be the homomorphism  which  carries an element
  $b$ in $B$ to an inner bimultiplication of
  $B$. Then $(B,M_B,\mu,id)$ is an E-system. In general, this E-system is not regular.

Standard consequences of the axioms of an E-system are as below.
\begin{pro}\label{md2}
Let $(B,D,d,\theta)$ be an E-system. \\
$\mathrm{i)}\;\; \mathrm{Ker}d\subset C_B$.\\
$\mathrm{ii)}\; \mathrm{Im}d$ is an ideal in $D$.\\
$\mathrm{iii)}$  The homomorphism $\theta$ induces a  homomorphism
$\varphi:D\ri M_{\mathrm{Ker}d}$ given by
$$\varphi_x=\theta_x|_{\mathrm{Ker}d}.$$
$\mathrm{iv)} \;\mathrm{Ker}d$ is a $\mathrm{Coker}d$-bimodule with
the actions
$$sa=\varphi_x(a),\ \ as=(a)\varphi_x, \ \  a\in\mathrm{ Ker}d, \ x\in s\in \mathrm{Coker}d.$$
\end{pro}

To state the relation between regular E-systems and crossed
bimodules over rings, one recalls the following definition.

\noindent{\bf Definition 4.} A functor $\Phi:\mathbf C\ri \mathbf
C'$ is an isomorphism of categories if  it is bijective on objects
and on morphism sets.

\begin{thm}\label{md1} The categories of regular E-systems and of crossed bimodules over rings are
isomorphic.
\end{thm}
\begin{proof}  Let $_{s}B=(B,D,d,\theta)$ be a regular E-system. The abelian additive group
 $B$ is a $D$-bimodule with the actions
\begin{equation}\label{td}
xb=\theta_xb,\;\;\;bx=b\theta_x,
\end{equation}
for $x\in D, b\in B.$ It is then easy to check that the axioms of a
crossed bimodule hold. For example, the relation \eqref{eq14}
follows from the relation \eqref{eq12},
\begin{equation} d(b)b'= \theta _{d(b)}(b')\stackrel{\eqref{eq12} }{=} \mu_{b}(b')= bb'
=b\mu_{b'}\stackrel{\eqref{eq12} }{=}b\theta_ {d(b')}=bd(b'), \notag
\end{equation}
since $\mu_{b}, \mu_{b'}$ are inner bimultiplications of the ring
$B$. Besides, the regularity of the E-system $(B,D,d,\theta)$ is
necessary and sufficient  for the two-sided module $B$ to be a
$D$-bimodule.

Conversely, if  $_{c}B=(B,D,d)$ is a crossed bimodule then $B$ has a
ring structure with the multiplication
 \begin{equation}\label{pn}
b\ast b':=d(b)b'=bd(b'),\;b,b'\in B.
 \end{equation}
Clearly, $d: B\rightarrow D$ is a ring homomorphism since  %$D$-bimodule nªn tr­íc hÕt $d$ lµ mét ®ång cÊu gi÷a c¸c nhãm céng.
 for all $b,b'\in B$,
 $$d(b\ast b')=d(d(b)b')=d(b)d(b').$$
The map $\theta:D\rightarrow M_B$ is defined by the $D$-bimodule
actions (\ref{td}). Then, $\theta$ is a homomorphism with image in
$M_B$, the elements of $\theta(D)$ are permutable  since $B$ is a
$D$-bimodule. The homomorphism $\theta$ satisfies the condition
 \eqref{eq13} since $d$ is a homomorphism of bimodules.
Thus, the correspondence   $_{s}B\mapsto\ _{c}B$ is bijective on
objects.

Now, if $(f_1,f_0):(B,D,d,\theta)\ri (B', D', d', \theta')$ is a
morphism of E-systems, it is then clear that $(f_1, f_0)$ satisfies
the relation (\ref{cb1}).

Further, for all $x\in D, b\in B$, one has
$$f_1(xb)\stackrel{(\ref{td})}{=}f_1(\theta_xb)\stackrel{(\ref{eq15})}{=}\theta'_{f_0(x)}f_1(b)\stackrel{(\ref{td})}{=}f_0(x)f_1(b)=xf_1(b).$$
Similarly, one obtains $f_1(bx)=f_1(b)x.$ This means that the pair
$(f_1,f_0)$ is a morphism of crossed bimodules.

 \indent Conversely, let
$(k_1,k_0):(B,D,d)\ri (B',D',d')$ be a morphism of crossed
bimodules. We show that $k_1$ is a ring homomorphism. According to
the determination of the multiplication on the ring $B$, we have
$$k_1(b\ast b')\stackrel{(\ref{pn})}{=}k_1(d(b)b')\stackrel{(\ref{cb2})}{=}k_0(d(b))k_1(b')\stackrel{(\ref{cb1})}{=}
d'(k_1(b))k_1(b')\stackrel{(\ref{pn})}{=}k_1(b)\ast k_1(b'),$$ for
all $b,b'\in B$. Besides, the pair $(k_1,k_0)$  also satisfies
(\ref{eq15}).
\end{proof}

By the above proposition, the notion of an E-system can be seen as a
weaken version of the notion of a crossed bimodule over rings.

We now discuss the relationship among the above concepts and  the
concept of a crossed module of $D$-structures in the category
$\mathbf C$ of $\Omega$-groups (see \cite{TP}). For convenience,
such a crossed module is called a crossed $\mathbf C$-module. T.
Porter proved that there is an equivalence between the category of
crossed $\mathbf C$-modules and that of internal categories in
$\mathbf C$. A crossed $\mathbf C$-module can be described as
follows.

\begin{pro}[Proposition 2 \cite{TP}] Given a $D$-structure on $B$,
$d:B\ri D$ is a crossed $\mathbf C$-module if and only if the
following conditions are satisfied for all $b,b_1,b_2\in B, x\in
D,\ast\in\Omega'_2\subset\Omega$

$\mathrm{(i).}$ $d((-x)\cdot b)=-x+d(b)+x$,

$\mathrm{(ii).}$ $(-d(b_1))\cdot b_2=-b_1+b_2+b_1$,

$\mathrm{(iii).}$ $d(b_1)\ast b_2=b_1\ast b_2=b_1\ast d(b_2)$,

$\mathrm{(iv).}$ $\begin{cases}d(xb)=x\ast d(b)\\ d(bx)=d(b)\ast
x.\end{cases}$
\end{pro}
\noindent Here $\ast$ is a binary operation which is not the group
operation  $+$, the actions $x\cdot b, x\ast b$ are given by
$$x\cdot b=s(x)+b-s(x),$$
$$x\ast b=s(x)\ast b,$$
where $s$ is the morphism in the split exact sequence
$$0\ri B\stackrel{i}{\ri}E\underset{s}{\stackrel{p}{\rightleftarrows}}D\ri
0.$$

To establish the link between these crossed $\mathbf C$-modules and
crossed modules over rings, we take $\mathbf C$ to be a category
whose objects are rings. The morphisms of $\mathbf C$ are ring
homomorphisms which are not necessarily 1-homomorphisms.

\begin{pro}
Every crossed $\mathbf C$-module is a crossed bimodule over rings.
\end{pro}
\begin{proof}
Let $d:B\ri D$ be a crossed $\mathbf C$-module. Then $d$ is a ring
homomorphism, and $D$ acts on $B$ by
\begin{equation}\label{pt}
xb=s(x)b,\ bx=bs(x),\ x\in D, b\in B.
\end{equation}
The map $\theta:D\ri M_B$ is given by
$$
\theta_x(b)=xb, (b)\theta_x=bx.
$$
Since $s$ is a ring homomorphism, so is $\theta$. The relation
(\ref{eq12}) follows from the condition  (iii). Indeed, for $b,b'\in
B$
$$(\theta d)(b)(b')=\theta_{db}(b')=(db)b'=bb'=\mu_b(b').$$
 It follows from (iv) that
$d(\theta_x(b))=d(xb)=xd(b)b.$ This means the relation (\ref{eq13})
holds, and therefore  $(B,D,d,\theta)$ is an E-system.
\end{proof}
One can see that a crossed $\mathbf C$-module $d:B\ri D$ satisfies
most of the conditions of a crossed bimodule over rings. We first
see that $B$ is a $D$-bimodule with the action \eqref{pt} By (iv),
the ring homomorphism $d:B\ri D$ is a $D$-bimodule. The relation
(\ref{eq14}) follows directly from the condition (iii). Note that
the ring $D$ is not necessarily unitary and if it has a unit, the
ring $B$ is not assumed to be a unitary $D$-bimodule. These
investigations show that the concept of a crossed $\mathbf C$-module
can be seen as a weaken version of the concept of a crossed bimodule
over rings.

\noindent{\bf Remark.} Since $\mathbf C$ can be any of categories of
$\Omega$-groups, use of crossed $\mathbf C$-modules has resulted in
various contexts. However, in each particular case there is a
certain restriction. For example, by Proposition 3 \cite{TP} Ker$d$
is singular; while for crossed modules over groups, (or crossed
modules over rings)  Ker$d$ is a subgroup in the center (or the
bicenter) of $B$.

Since rings with unit are not $\Omega$-groups, one can not seek a
relation among the category of  crossed $\mathbf C$-modules,
cohomology of algebras and cohomology of rings.

\section{ Strict Ann-categories and E-systems}
Crossed modules over groups are often studied in the form of strict
2-groups (see \cite{Baez, Br76, Br94}). In this section, we prove
that E-systems and strict Ann-categories are equivalent.

For every E-system $(B,D,d, \theta)$ we can construct a strict
Ann-category $\mathcal A=\mathcal A_{B\rightarrow D},$ called the
 Ann-category  {\it associated to} the
E-system $(B,D,d,\theta)$, as follows. One sets
$$\mathrm{Ob}(\mathcal A)=D,$$
 and for two objects $x,y$ of $\mathcal A$,
 $$\mathrm{Hom}(x,y)=\{b\in B/y=d(b)+x\}.$$
 The composition of  morphisms is given by
$$(x\stackrel{b}{\ri}y\stackrel{c}{\ri}z)=(x\stackrel{b+c}{\ri}z).$$
\noindent Two operations $\oplus, \otimes$ on objects are given by
the operations $+,\times$ on the ring $D$. For the morphisms, we set
$$(x\stackrel{b}{\rightarrow}y)\ts(x'\stackrel{b'}{\rightarrow}y')=(x+x'\stackrel{b+b'}{\longrightarrow} y+y'),$$
$$(x\stackrel{b}{\rightarrow}y)\tx(x'\stackrel{b'}{\rightarrow}y')=
(xx'\xrightarrow{bb'+b\theta_{x'}+\theta_xb'} yy').$$

Based on the definition of an E-system, it is easy to verify that
$\mathcal A$ is an Ann-category with the  strict constraints.

Conversely, for every strict Ann-category $(\mathcal A,\ts,\tx)$ one
can define an E-system $C_\mathcal A=(B,D,d,\theta )$. Indeed, let
\begin{equation}   D=\mathrm{Ob}(\mathcal A),\;\; B=\{0\xrightarrow{b}  x|\ x\in D\}.\notag \end{equation}
Then, $D$ is a ring with two operations
\begin{align*} x+y=x\ts y,\;\; xy=x\tx y,  \end{align*}
and $B$ is a ring with two operations
\begin{align*} b+c = b\ts c,\;\; bc=b\tx c. \end{align*}

The homomorphisms $d: B\rightarrow D$ and $\theta : D\rightarrow
M_B$ are defined by
\begin{align*} d(0\xrightarrow{b}  x)=x, \end{align*}
\begin{equation}  \theta _y(0\xrightarrow{b} x)=(0\xrightarrow{id_y \tx b}  yx), \notag \end{equation}
\begin{equation} (0\xrightarrow{b}  x)\theta _y=(0\xrightarrow{b\tx id_y}  yx). \notag \end{equation}
The quadruple  $(B,D,d,\theta )$ defined as above is an E-system.

In the following lemmas, let $\mathcal A_{B\ri D}$ and $\mathcal
A_{B'\ri D'}$ be Ann-categories associated to E-systems
$(B,D,d,\theta)$ and $(B',D',d',\theta')$, respectively.

\begin{lem}\label{t1}
Let $(f_1,f_0):(B,D,d,$ $\theta)\ri(B',D',d',\theta')$ be a morphism
of E-systems.

\noindent $\mathrm{i)}$ There is a functor $F:\mathcal A_{B\ri
D}\ri\mathcal A_{B'\ri D'}$
 defined by
$$F(x)=f_0(x),\ F(b)=f_1(b),\ x\in\mathrm{ Ob}(\mathcal A_{B\ri
D}), b\in\mathrm{ Mor}(\mathcal A_{B\ri D}).$$
 $\mathrm{ii)}$ The functor $F$ together with  isomorphisms
$\breve{F}_{x,y}: F(x+y)\ri Fx+Fy$, $\widetilde{F}_{x,y}:F(xy)\ri
FxFy$  is an Ann-functor if   $\breve{F}_{x,y}$ and
$\widetilde{F}_{x,y}$ are constants in $\mathrm{Ker}d'$ and for all
$x,y\in D$ the following conditions hold:
\begin{equation}\label{ttfn}\theta'_{Fx}(\widetilde{F})=(\widetilde{F})\theta'_{Fy}=\widetilde{F},\end{equation}
\begin{equation}\label{ttfa}\theta'_{Fx}(\breve{F})=(\breve{F})\theta'_{Fy}=\breve{F}+\widetilde{F}.\end{equation}
\end{lem}

Then, we say that $F$ is an Ann-functor of {\it form} $(f_1,f_0)$.
\begin{proof}

i) Every element $b\in B$ can be considered as a morphism
$(0\stackrel{b}{\ri} db)$ in $\mathcal A_{B\ri D}$. Then,
$(F0\stackrel{F(b)}{\ri} F(db))$ is a morphism in $\mathcal A_{B'\ri
D'}.$
 By the construction of the
Ann-category associated to an E-system, $F$ is a functor.

\noindent ii) We define the natural isomorphisms
$$\breve{F}_{x,y}: F(x+y)\rightarrow F(x)+ F(y),\; \widetilde{F}_{x,y}:F(xy)\ri F(x)F(y)$$
such that $F=(F,\breve{F},\widetilde{F})$ becomes an Ann-functor.
First we  see that
$$F(x)+F(x')=F(x+x'),$$
so $d'(\breve{F}_{x,x'})=0$. Analogously,
$d'(\widetilde{F}_{x,x'})=0,$ thus
\begin{equation}
\breve{F}_{x,x'},\widetilde{F}_{x,x'}\in \mathrm{Ker}d'\subset
C_{B'}.\label{eq1a}
\end{equation}
Now, for two morphisms $(x\stackrel{b}{\ri}y)$ and
$(x'\stackrel{b'}{\ri}y')$
 in $\mathcal A_{B\ri
D}$, we have:
\[ \bullet\ F(b\oplus
b')=F(x+x'\xrightarrow{b+b'}
y+y')=\big(f_0(x+x')\xrightarrow{f_1(b+b')} f_0(y+y')\big),\]
\[\begin{aligned}F(b)\oplus F(b')&=\big(f_0(x)\xrightarrow{f_1(b)} f_0(y)\big)\oplus
\big(f_0(x')\xrightarrow{f_1(b')} f_0(y')\big)\\
\;&=\big(f_0(x)+f_0(x')\xrightarrow{f_1(b)+f_1(b')}
f_0(y)+f_0(y')\big).\end{aligned}\] Since $f_1$ is a ring
homomorphism, one obtains
\begin{equation}F(b\oplus b')=F(b)\oplus F(b'). \label{eq16}
\end{equation}
By (\ref{eq1a}) and (\ref{eq16}), the commutativity of the diagram
\begin{equation}\label{bdfa}\begin{diagram}
\node{F(x+x')}\arrow{e,t}{\breve{F}_{x,x'}}\arrow{s,l}{F(b\oplus
b')}\node{F(x)+F(x')}\arrow{s,r}{F(b)\oplus F(b')}\\
\node{F(y+y')}\arrow{e,b}{\breve{F}_{y,y'}}\node{F(y)+F(y')}
\end{diagram}\end{equation}
follows from $\breve{F}_{x,x'}=\breve{F}_{y,y'}$.
\[ \bullet \ F(b\otimes
b')=F(xx'\xrightarrow{bb'+b\theta_{x'}+\theta_xb'}
yy')=\big(f_0(xx')\xrightarrow{f_1(bb'+b\theta_{x'}+\theta_xb')}
f_0(yy')\big),\]
\[\begin{aligned}F(b)\otimes F(b')&=\big(f_0(x)\stackrel{f_1(b)}{\longrightarrow}f_0(y)\big)\otimes
\big(f_0(x')\stackrel{f_1(b')}{\longrightarrow}f_0(y')\big)\\
\;&=\big(f_0(x)f_0(x')
\xrightarrow{f_1(b)f_1(b')+f_1(b)\theta'_{f_0(x')}+\theta'_{f_0(x)}f_1(b')}
f_0(y)f_0(y')\big).\end{aligned}\]
 \indent By
\eqref{eq15}, $f_1(\theta_xb')=\theta'_{f_0(x)}f_1(b')$ and
$f_1(b\theta_{x'})=f_1(b)\theta'_{f_0(x')}$, hence
\begin{equation}F(b\otimes b')=F(b)\otimes \label{eq17}
F(b').\end{equation}

By (\ref{eq1a}) and (\ref{eq17}), the commutativity of the diagram
\begin{equation}\begin{diagram}
\node{F(xx')}\arrow{e,t}{\widetilde{F}_{x,x'}}\arrow{s,l}{F(b\otimes
b')}\node{F(x)F(x')}\arrow{s,r}{F(b)\otimes F(b')}\\
\node{F(yy')}\arrow{e,b}{\widetilde{F}_{y,y'}}\node{F(y)F(y')}
\end{diagram}\label{bdfn}\end{equation}
follows from $\widetilde{F}_{x,x'}=\widetilde{F}_{y,y'}.$ The
equalities (\ref{ttfn}) and (\ref{ttfa}) come from the
 compatibility of  $(F,\widetilde{F})$ with the
associativity constraint and the distributivity ones, respectively.
\end{proof}
An Ann-functor $F$ is {\it single} if $F(0)=0', F(1)=1'$ and
$\breve{F},\widetilde{F}$ are constants. Then we state the converse
of Lemma \ref{t1}.

\begin{lem}\label{n1}
Let   $(F,\breve{F},\widetilde{F}):\mathcal A_{B\ri D}\ri\mathcal
A_{B'\ri D'}$ be a single Ann-functor. Then, there is a morphism of
E-systems $(f_1,f_0):(B\ri D)\rightarrow (B'\ri D')$, where
$$f_1(b)=F(b),\;\ f_0(x)=F(x),$$
for $b\in B, x\in D$.
\end{lem}

\begin{proof}

Since $F(0)=0', F(1)=1'$
 and  $\breve{F},\widetilde{F}$ are constants, it is easy to see that $\breve{F},\widetilde{F}$ are in Ker$d'$.
 By the determination of a morphism in $\mathcal A_{B'\ri
D'}$,
$$F(x+y)=F(x)+F(y),\;\;F(xy)=F(x)F(y), $$
so $f_0$ is a ring homomorphism.

Since $\breve{F}$ is a constant in Ker$d'$, the commutative diagram
(\ref{bdfa}) implies
$$F(b\oplus b')=F(b)\oplus F(b').$$
This means that $f_1(b+b')=f_1(b)+f_1(b')$.

Since $\widetilde{F}$ is a constant in Ker$d'$, the commutative
diagram (\ref{bdfn}) implies
$$F(b\otimes b')=F(b)\otimes F(b').$$
By the definition of $\otimes$,
\begin{equation}\label{ht22}f_1(bb')+f_1(b\theta_{x'})+f_1(\theta_xb')=f_1(b)f_1(b')+
f_1(b)\theta'_{f_0(x')}+\theta'_{f_0(x)}f_1(b').\end{equation} In
this relation, taking $b=0$ and then $ b'=0$ yield
$$f_1(\theta_x b')=\theta'_{f_0(x)}f_1(b'),\; f_1(b\theta _{x'}) = f_1(b)\theta '_{f_0(x')}.$$
Thus, \eqref{eq15}  holds. Then, the equation (\ref{ht22}) turns
into $f_1(bb')=f_1(b)f_1(b')$, that is,   $f_1$ is a ring
homomorphism. The rule (\ref{gr1}) also holds. Indeed, for all
morphisms $(x\stackrel{b}{\rightarrow}y)$ in $\mathcal A_{B\ri D},$
$y=d(b)+x.$ It follows that

\[
f_0(y)=f_0(d(b)+x)=f_0(d(b))+f_0(x).\] Besides,
$(f_0(x)\stackrel{f_1(b)}{\rightarrow}f_0(y))$ is a morphism in
$\mathcal A_{B'\ri D'}$, so
\[f_0(y)=d'(f_1(b))+f_0(x).\]
Thus, $f_0(d(b))=d'(f_1(b))$ for all $b\in B.$
\end{proof}

\begin{lem}\label{dl}
Two Ann-functors
$(F,\breve{F},\widetilde{F}),(G,\breve{G},\widetilde{G}):\mathcal
A_{B\ri D}\ri\mathcal A_{B'\ri D'}$  of the same form are homotopic.

\end{lem}
\begin{proof} Suppose that $F$ and $G$ are two Ann-functors of form $(f_1,f_0)$.
By Lemma \ref{t1}, $\breve{F}, \breve{G}$ are constants. We prove
that $\alpha=\breve{G}-\breve{F}$ is a homotopy between $F$ and $G$.
It is easy to check the naturality of $\alpha$ and the compatibility
of
 $\alpha$ with the addition. Besides, $\alpha$ is compatible with
 the multiplication. In other words, the following diagram
\begin{equation}\label{dlfg}\begin{diagram}
\node{F(xy)}\arrow{e,t}{\widetilde{F}}\arrow{s,l}{\alpha}\node{F(x)F(y)}\arrow{s,r}{\alpha\otimes \alpha}\\
\node{G(xy)}\arrow{e,b}{\widetilde{G}}\node{G(x)G(y)}
\end{diagram}\end{equation}commutes.
 Indeed, by Lemma
 \ref{t1},
\[\begin{aligned}\widetilde{G}-\widetilde{F}=&(\theta'_{Fx}(\breve{G})-\breve{G})
-(\theta'_{Fx}(\breve{F})-\breve{F})\\
=&\theta'_{Fx}(\alpha)-\alpha.\end{aligned}\]

Since $\alpha\in \mathrm{Ker}d'\subset C_{B'}$, so
\[\begin{aligned}\alpha\otimes
\alpha=&\alpha.\alpha+(\alpha)\theta'_{Gy}+\theta'_{Gx}(\alpha)\\
=&(\alpha)\theta'_{Gy}+\theta'_{Gx}(\alpha).\end{aligned}\] For
$y=0$, or $x=0$ we have
$$\alpha\otimes
\alpha=(\alpha)\theta'_{Gy}=\theta'_{Gx}(\alpha).$$ Thus,
$$\widetilde{G}-\widetilde{F}=\alpha\otimes
\alpha-\alpha,$$ that is, (\ref{dlfg}) holds.
\end{proof}

Two Ann-functors $(F,\breve{F}, \widetilde{F})$ and $ (G,\breve{G},
\widetilde{G})$ are {\it strong homotopic} if they are homotopic and
$F=G$. By Lemma (\ref{dl}), one obtains the following fact.

\begin{hq}\label{hq10} Two Ann-functors  $F,G:\mathcal A_{B\ri D}\ri \mathcal A_{B'\ri D'}$
are strong homotopic if and only if they are of the same form.
\end{hq}

We write
 ${\bf Annstr}$
for the   category of strict Ann-categories  and their single
Ann-functors.
 We can define the {\it strong homotopy category}
$Ho\mathbf{Annstr}$  to be the quotient category with the same
objects, but morphisms are strong homotopy classes of single
Ann-functors. We write
 ${\mathrm{Hom}}_{{\bf Annstr}}[\mathcal{A},\mathcal{A}'] $ for the homsets of the homotopy category, that is,
$${\mathrm{Hom}}_{{\bf Annstr}}[\mathcal{A},\mathcal{A}']=\frac{{\mathrm{Hom}}_{{\bf Annstr}}(\mathcal{A},\mathcal{A}')}
{\text{strong homotopies}}.$$ Denote  $\bf{ESyst}$  the category of
E-systems, we obtain the following result which is an extending of
Theorem 1 \cite{Br76}

\begin{thm}[Classification Theorem] \label{pldl} There exists an
equivalence
\[\begin{matrix}
 \Phi:& \bf{ESyst}&\ri&Ho{\bf Annstr}\\
&(B\ri D)&\mapsto&{\mathcal{A}_{B\ri D}}\\
&(f_1,f_0)&\mapsto&[F]
\end{matrix}\]
where $F(x)=f_0(x), F(b)=f_1(b)$, for $x\in \mathrm{Ob}\mathcal A,
b\in \mathrm{Mor}\mathcal A$.
\end{thm}
\begin{proof}
By Corollary  \ref{hq10}, the correspondence  $\Phi$  on homsets,
$$\mathrm{Hom}_{\mathbf{ESyst}}(B\ri D,B'\ri D')\rightarrow \mathrm{Hom}_{{\bf Annstr}}[\mathcal{A}_{B\ri D},
\mathcal{A}_{B'\ri D'}],$$ is an injection.  By Lemma \ref{n1},
every single Ann-functor  $F:\mathcal{A}_{B\ri D}\ri
\mathcal{A}_{B'\ri D'}$ determines a morphism of E-systems
$(f_1,f_0)$, and clearly $\Phi(f_1,f_0)=[F]$, thus $\Phi$ is
surjective on homsets.

Let $C_\mathcal A$ be an E-system associated to a strict
Ann-category $\mathcal A$. By the construction of an Ann-category
associated to an E-system, $\Phi(C_\mathcal A)=\mathcal A$ (rather
than an isomorphism).   Hence, $\Phi$ is an equivalence of
categories.
\end{proof}

%%%%%%%%%%%%%%%%
%\newpage

\section{Ring extensions  of the type of an E-system}
In this section we consider the ring extensions  of the type of an
E-system, which are analogous to the group extensions of the type of
a crossed module  \cite{Br94}.

\noindent\textbf{Definition 5.} Let $(B,D,d,\theta)$ be an E-system.
A ring \emph{extension }of  $B$ by $Q$ of \emph{type } $B\ri D$ is a
diagram of ring homomorphisms\begin{align*}  \begin{diagram}
\xymatrix{ 0 \ar[r]& B \ar[r]^j \ar@{=}[d] &E  \ar[r]^p \ar[d]^\varepsilon & Q \ar[r]& 0, \\
& B \ar[r]^d & D}
\end{diagram}
\end{align*}
 where the top row is exact, the quadruple $(B,E,j,\theta')$ is an
E-system where $\theta'$ is given by the bimultiplication type, and
the pair $(id,\varepsilon)$ is a morphism of E-systems.

Two extensions of $B$ by $Q$ of  type $B\xrightarrow{d}D$ are said
to be \emph{equivalent } if there is a morphism of exact sequences
\begin{align}\begin{diagram}\label{mrtd}
\xymatrix{ 0 \ar[r]& B \ar[r]^j \ar@{=}[d] &E  \ar[r]^p \ar[d]^\eta
& Q \ar[r]\ar@{=}[d] & 0,&\;\;\;E
\ar[r]^\varepsilon&D \\
 0 \ar[r]& B \ar[r]^{j'}  &E'  \ar[r]^{p'}   & Q \ar[r]& 0,&\;\;\;E'
\ar[r]^{\varepsilon'}&D}\end{diagram}
\end{align}
and $\varepsilon'\eta=\varepsilon$.
 Obviously, $\eta $ is an isomorphism.

In the diagram
\begin{align}  \begin{diagram}
\xymatrix{\mathcal E:\;\;\; 0 \ar[r]& B \ar[r]^j \ar@{=}[d] &E  \ar[r]^p \ar[d]^\varepsilon & Q \ar[r]\ar@{.>}[d]^\psi & 0, \\
 & B \ar[r]^d  &D  \ar[r]^q   & \text{Coker}d }
\end{diagram}\label{eq10}
\end{align}
where $q$ is a canonical projection, since the top row is exact and
$ q\circ \varepsilon \circ j= q\circ d =0,$ there  is a ring
homomorphism $\psi: Q\rightarrow \text{Coker}d $ such that the right
hand side square commutes. Moreover, $\psi$ depends only on the
equivalence class of the extension $\mathcal E$. Our purpose  is to
study the set
 $$\mathrm{Ext}_{B\ri D}(Q,B,\psi)$$
of equivalence classes of extensions of $B$ by $Q$ of type $B\ri D$
inducing $\psi$. The results use the obstruction theory of
Ann-functors

Let $\mathcal A= \mathcal A_{B\rightarrow D}$ be the Ann-category
associated to an E-system $B\rightarrow D$. Clearly, $\pi_0\mathcal
A=\mathrm{Coker}d$,  $\pi_1\mathcal A=\mathrm{Ker}d$ and therefore
the reduced Ann-category $S_{\mathcal A}$ is of  form%Khi ®ã, Ann-ph¹m trï thu gän $S_{\mathcal A}$ lµ chÆt chÏ.
$$S_\mathcal A= (\mathrm{Coker} d, \mathrm{Ker} d,k),$$
where $\overline{k}\in H^3_{Shu} (\mathrm{Coker} d, \mathrm{Ker} d)$
since $\mathcal A$ and $S_\mathcal A$ are regular Ann-categories.
The homomorphism $\psi:Q\rightarrow \mathrm{Coker} d$ induces an
 {\it obstruction},
\begin{equation}\psi^\ast k\in Z^3_{Shu} (Q, \mathrm{Ker} d),\label{cct}
\end{equation}
which plays a  fundamental role to state  Theorem  \ref{dlc}. This
is the main result of this section, an extending of Theorem 5.2
\cite {Br94}. Besides, a particular case of a regular E-system  when
$Q=\mathrm{Coker} d$ and $\psi=id_{\mathrm{Coker}d}$ is a {\it
$\partial$-extension}  \cite{P2004}, so our result contains Theorem
4.4.2 \cite {P2004}.
\begin{thm}\label{dlc} Let $(B,D,d,\theta)$ be a regular E-system,
$\psi: Q\rightarrow \mathrm{Coker} d$ be a ring homomorphism. Then,
the vanishing of $\overline{\psi^\ast k}$   in $H^3_{Shu} (Q,
\mathrm{Ker} d)$ is necessary and sufficient for there to exist a
ring extension of $B$ by $Q$ of type $B\rightarrow D$ inducing
$\psi$. Further, if $\overline{\psi^\ast k}$ vanishes then there is
a bijection
$$\mathrm{Ext}_{B\ri D}(Q,B,\psi)\leftrightarrow
H^2_{Shu}(Q,\mathrm{Ker}d).$$
\end{thm}

 The first assertion is based on the
following lemmas.
\begin{lem}\label{mrlk}
For every Ann-functor $(F,\breve{F}, \widetilde{F}):\mathrm{Dis}
Q\ri \mathcal A$ there exists an extension $\mathcal E_F$ of $B$ by
$Q$ of type $B\ri D$ inducing $\psi: Q\ri \mathrm{Coker} d$.
\end{lem}
Such extension $\mathcal E_F$ is called an {\it associated }
extension to Ann-functor  $F$.

\begin{proof}
By Proposition \ref{md22}, $(F,\breve{F}, \widetilde{F})$ induces an
Ann-functor $K:\mathrm{Dis} Q\ri S_\mathcal A$ of type $(\psi,0)$.
Let $(H,\breve{H}, \widetilde{H}):S_\mathcal A\ri \mathcal A$ be a
canonical Ann-functor defined by the stick $(x_s, i_x)$. By \eqref
{eq2}, we have
$$ H(s)= x_s,\; H(s,b)=b,\;\breve{H}_{s,r}=-i_{x_s+ x_r},\; \widetilde{H}_{s,r}=-i_{x_s\cdot
x_r}.$$ Also by Proposition \ref{md22}, $(F,\breve{F},
\widetilde{F})$ is homotopic to  the composition
$$ \text{Dis} Q\xrightarrow{K} S_\mathcal A\xrightarrow{H}\mathcal A. $$
So one can choose $(F,\breve{F}, \widetilde{F})$ being  this
composition. By the determination of $\breve{HK}$ and
$\widetilde{HK}$,

\begin{equation}\label{eq24}\breve{F}_{u,v}= f(u,v)= f'(u,v)-i_{x_s+x_r},\end{equation}
\begin{equation}\label{eq25} \widetilde{F}_{u,v} = g(u,v)= g'(u,v) -i_{x_s\cdot  x_r}\in B, \end{equation}
where $ u,v\in Q, s=\psi(u), r= \psi(v), f'(u,v)=\breve{K}_{u,v},
g'(u,v)=\widetilde{K}_{u,v}.$ By the compatibility of $(F,\breve{F}
,\widetilde{F})$ with the strict constraints of Dis$Q$ and $\mathcal
A$, the functions $f$ and $g$ are the ``normal" ones satisfying

\begin{equation} f(u,v+t)+f(v,t)-f(u,v)-f(u+v,t)=0,\label{eq255}
\end{equation}
\begin{equation}
f(u,v)=f(v,u),
\end{equation}
\begin{equation}
\theta _{Fu}g(v,t)-g(uv,t)+g(u,vt)-g(u,v)\theta _{Ft}=0,
\label{eq27}
\end{equation}
\begin{equation}
g(u,v+t)-g(u,v)-g(u,t)+\theta _{Fu}f(v,t)-f(uv,ut)=0,
\end{equation}
\begin{equation}
g(u+v,t)-g(u,t)-g(v,t)+f(u,v)\theta _{Ft}-f(ut,vt)=0.
\end{equation}
The function $\varphi : Q\rightarrow M_B$ defined by
\begin{equation}\label{vp}\varphi (u)= \theta _{Fu} = \theta _{x_s}\;\;
(s=\psi(u)) \notag\end{equation} satisfies the relations
\begin{equation}
\varphi (u)+\varphi (v)=\mu_{f(u,v)}+ \varphi (u+v), \label{eq31}
\end{equation}
\begin{equation}
\varphi (u)\varphi (v)=\mu_{g(u,v)} + \varphi (uv). \label{eq32}
\end{equation}
We only prove the relation \eqref{eq31}, the proof of \eqref{eq32}
follows from \eqref{eq25} in the same way. Since
$f'(u,v)=\breve{K}_{u,v}\in\mathrm{ Ker}d$, then by Proposition
\ref{md2}, $f'(u,v)\in C_B.$ By \eqref{eq24}, one has
$\mu_{f(u,v)}=\mu (-i_{x_s+x_r})$. Thus,
\begin{align} \varphi (u)+\varphi (v)&= \theta_{x_s}+ \theta_{x_r} = \theta _{x_s+ x_r} \notag{}\\
&= \theta [d(-i_{x_s+x_r})+x_{s+r}]\notag{}
= \theta [d(-i_{x_s+ x_r})]+\theta_{x_{s+r}}\notag{}\\
&= \mu(-i_{x_s+ x_r})+ \varphi (u+v)\notag{}
\stackrel{\eqref{eq24}}{=} \mu_{f(u,v)}+ \varphi (u+v).  \notag
\end{align}

Since the family of functions $(\varphi , f, g)$ satisfies the
relations \eqref{eq255} - \eqref{eq32}, we have a crossed product
%in the sense of \cite{Mac2}
$ E_0=[B,\varphi,f,g,Q],$ that means $ E_0=B\times Q$, and two
operations are
\begin{equation}    (b,u)+(b',u')= (b+ b' + f(u,u'),u+u'), \notag \end{equation}
\begin{equation}  (b,u).(b',u')= (b.b'+b\varphi (u')+\varphi (u)b' + g(u,u'),uu'). \notag
\end{equation}
The set $E_0$ satisfies the axioms of a ring, in which note that the
associativity for the multiplication in  $E_0$ holds if and only if
the E-system $B\ri D$ is regular. Indeed, one can calculate the
triple products as follows:
\begin{align*}
[(b,u)(b',u')](b'',u'') &=((bb')b''+
b\varphi(u')\varphi(u'')+[\varphi(u)b']
\varphi(u'')\\
&+g(u,u')\varphi(u'') +\varphi(uu')b''+g(uu',u''),(uu')u''),
 \end{align*}
\begin{align*} (b,u)[(b',u')(b'',u'')]
&=(b(b'b'')+ b\varphi(u'u'')+\varphi(u)[b'\varphi(u'')]\\
&+\varphi(u)\varphi(u')b'' +\varphi(u)g(u,u')+g(u,u'u''),u(u'u'')),
\end{align*}
By \eqref{eq27}, \eqref{eq32}, associative law for the
multiplication in $B,Q,$ and commutative law for the addition in
$B,$ especially by the relation (\ref{eq133}),
$[\varphi(u)b']\varphi(u'')=\varphi(u)[b'\varphi(u'')],$ we get the
associative law for product in $E_0.$ Then, there is an exact
sequence of ring homomorphisms
$$\mathcal E_F:\;\;0\rightarrow B\stackrel{j_0}{\rightarrow}E_0\stackrel{p_0}{\rightarrow}Q\rightarrow 0,$$
where $j_0(b)=(b,0);\ p_0(b,u)=u,\ b\in B, u\in Q.$ Since $j_0(B)$
is a two-sided ideal in $E_0$, $B\xrightarrow{j_0}E_0$ is an
E-system, where $\theta_0: E_0\rightarrow M_B$ is given by the
bimultiplication type.

We define a ring homomorphism $\varepsilon:E_0\ri D$ by
$$\varepsilon(b,u)=db+x_{\psi(u)},\;(b,u)\in E_0,$$
where $x_{\psi(u)}$ is a representative of $u$ in $D$. We  show that
the pair $(id_B,\varepsilon)$ satisfies the rules \eqref{gr1},
\eqref{eq15}. Clearly, $\varepsilon\circ j_0=d$. Besides, for all
$(b,u)\in E_0,c\in B$,
$$\theta_0(b,u)(c)=j^{-1}_0[(b,u)(c,0)]=bc+\varphi(u)c,$$
$$\theta_{\varepsilon(b,u)}(c)=\theta_{db+x_{\psi(u)}}c=bc+\varphi(u)c.$$
Thus, $\theta_0(b,u)(c)=\theta_{\varepsilon(b,u)}(c)$. Analogously,
$c\theta_0(b,u)=c\theta_{\varepsilon(b,u)}$. So $(id_B,\varepsilon)$
is a morphism of E-systems, that is, one has an extension
\eqref{eq10}, where $E$ is replaced by $E_0$.

For all $u\in Q$ we have $ q\varepsilon(0,u)= q(x_{\psi(u)})=
\psi(u)$, then the extension  $\mathcal E_F$ induces $\psi:
Q\rightarrow \text{Coker} \ d$.
\end{proof}

{\bf The proof of Theorem \ref{dlc}}
\begin{proof}
Let us recall that $\mathcal A$ is the Ann-category associated to
the regular E-system $B\stackrel{d}{\rightarrow}D$. Then, its
reduced Ann-category is $S_{\mathcal A}=(\mathrm{Coker} d,
\mathrm{Ker} d, k)$, where $k\in Z^{3}_{Shu}(\mathrm{Coker}
d,\mathrm{Ker }d)$.  The pair
$$(\psi,0):(Q,0,0)\ri (\mathrm{Coker} d, \mathrm{Ker }d,k)$$
has $-\psi^*k$ as an obstruction. By the assumption,
$\overline{\psi^\ast k}=0$, hence by Proposition \ref{s63} the pair
$(\psi,0)$ determines
an Ann-functor $(\Psi,\breve{\Psi},\widetilde{\Psi}):\mathrm{Dis} Q\ri S_{\mathcal A}$. %, víi $\widetilde{\Psi}=id$.
 Then the composition of $(\Psi,\breve{\Psi}, \widetilde{\Psi})$ and
$(H,\breve{H}, \widetilde{H}): S_{\mathcal A}\ri \mathcal A$ is an
Ann-functor $(F,\breve{F}, \widetilde{F}):\mathrm{Dis} Q\ri \mathcal
A$, and by Lemma \ref{mrlk} we obtain an associated extension
$\mathcal E_F$.

Conversely, suppose that there is an extension as in the diagram
(\ref{eq10}). Let $ \mathcal A'$  be the Ann-category associated to
the E-system $B\ri E$. By Proposition \ref{md22}, there is an
Ann-functor $F:\mathcal A' \ri \mathcal A$. Since the reduced
Ann-category of $\mathcal A'$ is $\mathrm{Dis} Q$, so  by
Proposition \ref{md22}, $F$ induces an Ann-functor of type
$(\psi,0)$ from $\mathrm{Dis} Q$ to $(\mathrm{Coker} d,\mathrm{Ker}
d,k)$.  Now, by Proposition \ref{s63}, the obstruction of the pair
$(\psi,0)$  must vanish in $H^3_{Shu}(Q, \mathrm{Ker} d),$ that is,
$\overline{\psi^\ast k}=0$.
\end{proof}

The final assertion of Theorem \ref{dlc} follows from the next
theorem.

\begin{thm} [Schreier Theory for ring extensions of the type of an E-system]\label{schr}
There is a bijection
$$ \Omega:\mathrm{Hom}^{Ann}_{(\psi,0)}[\mathrm{Dis}Q,\mathcal A] \rightarrow \mathrm{Ext}_{B\ri D}(Q,
B,\psi).
$$
\end{thm}
\begin{proof}

{\it Step 1: The Ann-functors $(F,\breve{F} ,\widetilde{F}) $,
$(F',\breve{F'} ,\widetilde{F}') $ are homotopic if and only  if
their corresponding associated extensions $\mathcal E_F, \mathcal
E_{F'}$ are equivalent.}

Let two Ann-functors $F, F':$ Dis$Q\ri\mathcal A$ be homotopic by a
homotopy
  $\alpha :F\ri F'$.
Then, by the definition of an Ann-morphism, the following diagrams
commute
\[\begin{diagram}
\node{F(u+v)}\arrow{e,t}{\breve{F} _{u,v}}\arrow{s,l}{\alpha_{u+
v}}\node{F(u)+F(v)}\arrow{s,r}{\alpha_u+\alpha_v}\\
\node{F'(u+v)}\arrow{e,b}{\breve{F'} _{u,v}}\node{F'(u)+F'(v),}
\end{diagram}\;\;\;\;\;\;\;\;
\begin{diagram}
\node{F(uv)}\arrow{e,t}{\widetilde{F}_{u,v}}\arrow{s,l}{\alpha_{uv}}\node{F(u)F(v)}\arrow{s,r}{\alpha_u\otimes
 \alpha_v}\\
\node{F'(uv)}\arrow{e,b}{\widetilde{F'}_{u,v}}\node{F'(u)F'(v).}
\end{diagram}\]
By the definition of the operation $\tx$ on $\mathcal A$,
$$\alpha_u\tx \alpha_v= \alpha_u\alpha_v+\alpha_u\theta _{Fv}+ \theta _{Fu}\alpha_v.$$
Then, since $f(u,v)=\breve{F}_{u,v}, f'(u,v)=\breve{F'}_{u,v},
g(u,v)=\widetilde{F}_{u,v}, g'(u,v)=\widetilde{F}'_{u,v}$, we have
\begin{equation}    f'(u,v)-f(u,v)=\alpha_u-\alpha_{u+v}+
\alpha_v, \label{eq40} \end{equation}
   \begin{equation}  g'(u,v)-g(u,v)= \alpha_u\alpha_v+\alpha_u
   \theta _{Fv}+ \theta _{Fu}\alpha_v- \alpha_{uv}.\label{eq41}  \end{equation}

Now, we set
\begin{align*} \alpha^*: E_F&\rightarrow E_{F'}\\
(b,u)&\mapsto (b-\alpha_u,u). \end{align*}

Note that $\theta _{F'u}=\mu_{\alpha_u}+\theta _{Fu}$, and by the
relations (\ref{eq40}), \eqref{eq41},  the correspondence $\alpha^*$
is an isomorphism. Besides,  the  diagram \eqref{mrtd} commutes in
which $E$ and $E'$ are replaced by $E_F$ and $E_{F'}$, respectively.

Finally,  $\varepsilon'\alpha^*=\varepsilon$. Indeed, since $\alpha
:F\ri F'$ is a homotopy, then  $Fu=x_{\psi(u)}=F'u$. Thus
$x_{\psi(u)}=d(\alpha _u)+ x_{\psi(u)}$, or $d(\alpha _u)=0$. Hence,
\begin{align*} \varepsilon'\alpha^*(b,u)=\ & \varepsilon'(b-\alpha _u,u)=d(b-\alpha _u)+
x_{\psi(u)}\\
= \ & d(b)-d(\alpha
_u)+x_{\psi(u)}=d(b)+x_{\psi(u)}=\varepsilon(b,u).
 \end{align*}
That means two extensions $\mathcal E_F$ and $\mathcal E_{F'}$ are
equivalent.

Conversely, if $\mathcal E_F$ and $\mathcal E_{F'}$ are equivalent,
there exists a ring isomorphism $(b,u)\mapsto (b-\alpha _u,u)$.
Then, we have a homotopy  $\alpha : F\rightarrow F'$  by retracing
our steps.

 {\it Step 2: $\Omega$ is a surjection.}

Let $\mathcal E$ be an extension $E$ of $B$ by $Q$ of type
$(B,D,d,\theta )$ inducing $\psi: Q\rightarrow \Coker d$ (see the
commutative diagram \eqref{eq10}). We prove that $\mathcal E$ is
equivalent to an extension $\mathcal E_F$ which is associated to an
Ann-functor $(F,\breve{F},\widetilde{F}):\mathrm{Dis} Q\ri \mathcal
A$.

Let  $\mathcal A'=\mathcal A_{B\rightarrow E}$ be the Ann-category
associated to the E-system $(B,E,j,\theta')$. By Lemma \ref{t1}, the
pair $(id_B,\varepsilon)$ in the diagram \eqref{eq10} determines a
single Ann-functor $(K,\breve{K}, \widetilde{K}):\mathcal
A'\ri\mathcal A$.

Since $\pi_0\mathcal A'=Q, \pi_1\mathcal A'=0$, the reduced
Ann-category $S_{\mathcal A'}$   is nothing else but the
Ann-category Dis$Q$. Choose a stick  $(e_u, i_e), e\in E, u\in Q$,
of $\mathcal A'$  (that is, $\left\{e_u\right\}$ is a representative
of $Q$ in
 $E$). By \eqref{eq2}, the canonical Ann-functor
 $(H',\breve{H'}, \widetilde{H}'):\mathrm{Dis} Q\ri
\mathcal A'$ is given by
$$H'(u)=e_u, \
\breve{H'}_{u,v}= -i_{e_u+e_v}= g'(u,v), \
\widetilde{H}'_{u,v}=-i_{e_u.e_v}=h'(u,v).$$ The composition
$F=K\circ H'$ is an Ann-functor $\mathrm{Dis} Q\ri \mathcal A$,
where
$$F(u)=\varepsilon(e_u), \
\breve{F}_{u,v}= \breve{H'}_{u,v}= g'(u,v), \
\widetilde{F}_{u,v}=\widetilde{H}'_{u,v}=h'(u,v).$$
 According to the
proof of Theorem \ref{dlc}, we construct an extension $\mathcal E_F$
of the crossed product $E_0=[B,\varphi,g',h',Q]$ which is associated
to $(F, \breve{F}, \widetilde{F})$.

 We now prove that $\mathcal E$
and $\mathcal E_F$ are equivalent, that is, there is a commutative
diagram

\begin{align*}  \begin{diagram}
\xymatrix{\mathcal E_F:\;\;\; 0 \ar[r]& B \ar[r]^{j_0} \ar@{=}[d]
&E_0 \ar[r]^{p_0} \ar[d]^\eta  & Q \ar[r]\ar@{=}[d] & 0&\;\;\;\;E_0
\ar[r]^{\varepsilon_0}&D \\
\mathcal E:\;\;\; 0 \ar[r]& B \ar[r]^{j}  &E \ar[r]^{p}   & Q
\ar[r]& 0&\;\;\;\;E \ar[r]^{\varepsilon}&D}
\end{diagram}
\end{align*}
and $\varepsilon\eta =\varepsilon_0.$

Indeed, since every element of $E$ can be written uniquely as
$b+e_u, b\in B$, we can define a map
$$\eta :E_0\ri E, \ (b,u)\mapsto b+e_u.$$
We next verify that  $\eta $ is a ring isomorphism. The
representatives  ${e_u}$  have the following properties
\begin{equation}
\varphi(u)c=\theta'_{e_u}c,\ c\varphi (u)=c\theta'_{e_u},  \ c\in
B,\label{23}
\end{equation}
\begin{equation}
e_u+e_v=-i_{e_u+e_v}+e_{u+v}=g'(u,v)+e_{u+v},\label{24}
\end{equation}
\begin{equation}
e_u.e_v =-i_{e_u.e_v}+ e_{u.v}= h'(u,v)+e_{uv}.\label{25}
\end{equation}
(The relation \eqref{23}  holds since the pair $(id_B,\varepsilon)$
is a morphism of E-systems. The relations \eqref{24}, \eqref{25}
hold thanks to the definition of a morphism in $\mathcal A'$.) Now,
we have
\begin{align*}
\eta[(b,u)+(c,v)]&=\eta(b+c+g'(u,v),u+v)=
b+c+g'(u,v)+e_{u+v}\\
&\stackrel{(\ref{24})}{=}b+c+e_u+e_v= (b+e_u)+(c+e_v)=\eta
(b,u)+\eta (c,v).
\end{align*}
\begin{align*}
\eta[(b,u)(c,v)]&=
\eta(bc+b\varphi(v)+\varphi(u)c+ h'(u,v),uv)\\
&=bc+b\varphi(v)+\varphi(u)c+ h'(u,v)+e_{uv}\\
&\stackrel{(\ref{23}),(\ref{25})}{=}bc+b\theta '_{e_v}+\theta '_{e_u}c + e_ue_v\\
&= bc+ b.e_v + e_u.c +e_u.e_v\\
&= (b+e_u).(c+e_v)=\eta (b,u).\eta (c,v).
\end{align*}
Finally, choose the representative $e_u$ such that
$\varepsilon(e_u)=x_{\psi(u)}$ (since it follows from \eqref{eq10}
that $q(\varepsilon(e_u))=\psi p(e_u)=\psi(u)$). Thus,
$$\varepsilon \eta (b,u)=\varepsilon (b+e_u)=\varepsilon(b)+\varepsilon(e_u)=d(b)+x_{\psi(u)}=\varepsilon_0(b,u),$$
that is, $\mathcal E$ and $\mathcal E_F$ are equivalent.
\end{proof}

Now, the bijection mentioned in Theorem  \ref{dlc} is obtained as
follows. Note that there is a natural bijection
$$\mathrm{Hom}[\mathrm{Dis}Q, \mathcal A]\leftrightarrow \mathrm{Hom}[\mathrm{Dis} Q, S_{\mathcal A}].$$
Then, since $\pi_0(\mathrm{Dis}Q)=Q$ and $ \pi_1(S_{\mathcal
A})=\mathrm{Ker}d$,  Theorem \ref{schr} and Theorem \ref{s63} imply
$$\mathrm{Ext}_{B\ri D}(Q, B,\psi)\leftrightarrow
H^2_{Shu}(Q,\mathrm{Ker }d).$$

{\bf Acknowledgement} The authors are much indebted to the referee, whose useful observations greatly improved our exposition.

\end{document}